\documentclass{article}[12pt]
\usepackage[utf8]{inputenc}
\usepackage{authblk}
\usepackage{amsmath, amssymb, amsthm, amsfonts, gensymb, tikz, commath, float, mathtools, enumerate, breqn, circuitikz}
\usetikzlibrary{positioning, arrows, arrows.meta, cd}
\usetikzlibrary{shapes,backgrounds,positioning,petri,topaths,calc}

\usepackage[all]{xy}
\usepackage{tabularx}
\usepackage{multirow}
\usepackage{esvect}
\usepackage{subcaption}
\usepackage{stmaryrd} 

\usepackage[text={15cm,23cm}]{geometry} 
\usepackage[colorlinks=true,%
            linkcolor=red!50!black,%
            citecolor=blue!50!black,%
            urlcolor=darkgray]{hyperref}  
\theoremstyle{plain}
\newtheorem{lem}{Lemma}[section]
\newtheorem{thm}[lem]{Theorem}
\newtheorem{cor}[lem]{Corollary}

\theoremstyle{definition}
\newtheorem*{rem}{Remark}
\newtheorem{ex}[lem]{Example}

\newtheorem{defn}[lem]{Definition}

\newcommand{\Z}{\mathbb{Z}}

\newcommand{\N}{\mathbb{N}}

\newcommand{\Q}{\mathbb{Q}}

\newcommand{\F}{\mathcal{F}}
\newcommand{\G}{\mathcal{G}}

\newcommand{\cM}{\mathcal{M}}
\newcommand{\cR}{\mathcal{R}}

\newcommand{\Sc}{\mathcal{S}}

\newcommand{\PSL}{\mathrm{PSL}}

\title{$q$-rationals and dimers}

\author{Valentin Ovsienko}

\date{}

\begin{document}

\maketitle

\begin{abstract}

We describe the relationships between the notion of $q$-deformed rational numbers, 
introduced in our previous work with Sophie Morier-Genoud, and the theory of dimer models.
We show that $q$-deformed rationals can be calculated
in terms of perfect matchings of certain bipartite graphs,
known as snake graphs, or ribbon tiles, etc.
equipped with a certain weight function on the set of edges.
We apply some elements of the dimer theory to get more information about $q$-rationals.

\end{abstract}

\section{Introduction and the main result}

The notion of $q$-deformed rational numbers, or $q$-rationals, was introduced
in~\cite{MGOfmsigma}.
This $q$-deformation can be understood as a map
\begin{equation}
\label{QMap}
[\,.\,]_q:\Q\cup\{\infty\}\to\Z(q)
\end{equation}
associating a rational function in one (formal) variable~$q$ with each rational number.
The map~\eqref{QMap} is determined by the property of
invariance (or equivariance) with respect to the action of the modular group $\PSL(2,\Z)$,
and the value at one point, say $[0]_q=0$.
For $n\in\N$, the $q$-deformation is given by the familiar formula
\begin{equation}
\label{EQInt}
\left[n\right]_q=
1+q+q^2+\cdots+q^{n-1}=\frac{1-q^n}{1-q},
\end{equation}
that goes back to Euler and Gauss.
When $\frac{r}{s}\in\Q$ such that $\frac{r}{s}\geq1$, one has
$$
\left[\frac{r}{s}\right]_q=\frac{\cR(q)}{\Sc(q)},
$$
where $\cR$ and $\Sc$ are monic polynomials with
positive integer coefficients.
Combinatorial properties of $q$-rational numbers have been studied by several authors; 
see, e.g.~\cite{BBL,FQ,Las,MPS,Ogu,OgRa,Ove,Tho}.
The notion of $q$-deformed rationals was extended to irrationals in~\cite{MGOexp},
first in terms of formal power series in~$q$, but later as analytic functions; see~\cite{Eti}.
For the analytic part of the theory; see~\cite{Eti,LMGOV,OP,OU}.
For a survey; see~\cite{MGOSur}.

The goal of the present paper is to link the recent theory 
of $q$-rationals to the well developed theory of dimer models.
We recover the polynomials $\cR$ and $\Sc$ of a $q$-deformed rational $\left[\frac{r}{s}\right]_q$
as the number of perfect matchings (or dimer coverings) with weighted edges
of a certain bipartite graph associated with~$\frac{r}{s}$.
This paper continues the recent preprint~\cite{EJMO} where
a similar combinatorial model was elabotrated for $q$-deformed Markov numbers.

Every rational number $\frac{r}{s}$ is associated (see~\cite{CR,CS}) 
with a graph called {\it snake graph} $\G_{\frac{r}{s}}$
(also known as ``ribbon polyomino'', ``ribbon tile''~\cite{Pak}, ``rim hook''~\cite{SW}, 
and ``border strip''; see~\cite{Sta}, p.345),
the ``snake terminology''  goes back to Propp~\cite{Pro}.
This is a connected bipartite graph in the plane
with vertical and horizontal edges
obtained by gluing elementary square boxes along edges:
a new box is attached to the previous one at the right or at the top.
For details; see Section~\ref{RatSec}.
For an example of the snake graph associated to
$\frac{179}{74}$; see Figure~\ref{SnakeGrid}
and other figures throughout the paper.

Given a rational $\frac{r}{s}$,
a beautiful theorem of \c{C}anak\c{c}i and Schiffler~\cite{CR} states that
the numerator~$r$
can be calculated as  the number of perfect matchings
of the corresponding snake graph $\G_{\frac{r}{s}}$.
The statement for the denominator~$s$ is similar and uses a smaller snake graph.

We extend the result of \c{C}anak\c{c}i and Schiffler to the $q$-deformed case
using snake graphs with weighted edges.
Note that a different way to calculate $q$-rationals with the help of
snake graphs and matrix realization of continued fractions was suggested in~\cite{BOSZ}.

\begin{defn}
\label{MainDefn}
Our rule to assign weights to the edges of snake graphs is the following.

\begin{enumerate}

\item
Consider the $\Z^2$-grid colored with two colors in the following  way.
\begin{figure}[H]
\begin{center}
\begin{tikzpicture}[scale=0.7]
\draw[step=1.0,black] (0.5,0.5) grid (9.5,6.5);
\draw[line width=2pt, blue ] (1,5) -- (1,6) ;
\draw[line width=2pt, red ] (1,4) -- (1,5)  (2,5) -- (2,6) -- (3,6);
\draw[line width=2pt, blue ] (1,3) -- (1,4)  (2,4) -- (2,5) -- (3, 5) -- (3,6) -- (4,6);
\draw[line width=2pt, red ] (1,2) -- (1,3)  (2,3) -- (2,4) -- (3,4) -- (3,5)-- (4,5)-- (4,6)-- (5,6);
\draw[line width=2pt, blue ] (1,1) -- (1,2)  (2,2) -- (2,3) -- (3, 3) -- (3,4) -- (4,4) -- (4,5) -- (5,5) -- (5,6)-- (6,6);
\draw[line width=2pt, blue ] (2,1) -- (3,1) -- (3,2) -- (4,2) -- (4, 3) -- (5,3) -- (5,4) -- (6,4) -- (6,5) -- (7,5)-- (7,6)-- (8,6);
\draw[line width=2pt, red ] (2,1) -- (2, 2) -- (3,2) -- (3,3) -- (4,3) -- (4,4) -- (5,4)-- (5,5)-- (6,5)-- (6,6)-- (7,6);
\draw[line width=2pt, red ] (3,1) -- (4,1) -- (4,2) -- (5,2) -- (5,3) -- (6,3) -- (6,4)-- (7,4)-- (7,5)-- (8,5)-- (8,6)-- (9,6);
\draw[line width=2pt, blue ] (4,1) -- (5,1) -- (5,2) -- (6,2) -- (6, 3) -- (7,3) -- (7,4) -- (8,4) -- (8,5)-- (9,5)-- (9,6);
\draw[line width=2pt, red ] (5,1) -- (6,1) -- (6,2) -- (7,2) -- (7,3) -- (8,3) -- (8,4)-- (9,4)-- (9,5);
\draw[line width=2pt, blue ] (6,1) -- (7,1) -- (7,2) -- (8,2) -- (8, 3)-- (9,3)-- (9,4);
\draw[line width=2pt, red ] (7,1) -- (8,1) -- (8,2)-- (9,2)-- (9,3);
\draw[line width=2pt, blue ] (8,1) -- (9,1)-- (9,2);
\draw[thick, ->] (1,6) -- (1,7);
\draw[line width=0.7pt] (1,0.5) -- (1,1);
\draw[line width=0.7pt] (0.5,1) -- (2,1);
\draw[thick, ->] (9,1) -- (10,1);
\end{tikzpicture}
\caption{Colored $\Z^2$-grid.}
\label{TheGrid}
\end{center}
\end{figure}

\noindent
(Observe that the horizontal segments in the first column are not colored.)

\item
Draw a snake graph in the grid with the down left vertex in the origin:

\begin{figure}[H]
	\centering
	\begin{tikzpicture}[scale=0.7]
\draw[step=1.0,black] (0.5,0.5) grid (9.5,6.5);
		\draw[line width=2pt] (1,1)-- (1,2);
		\draw[line width=2pt] (1,1)-- (4,1);
		\draw[line width=2pt] (1,2)-- (3,2);
		\draw[line width=2pt] (3,2)-- (3,4);
		\draw[line width=2pt] (4,1)-- (4,3);
		\draw[line width=2pt] (4,3)-- (8,3);
		\draw[line width=2pt] (3,4)-- (7,4);
		\draw[line width=2pt] (8,3)-- (8,6);
		\draw[line width=2pt] (7,4)-- (7,6);
		\draw[line width=2pt] (7,6)-- (8,6);
	
			\end{tikzpicture}
\caption{A snake graph $\G_{\frac{179}{74}}$ embedded  into $\Z^2$-grid.}
\label{SnakeGrid}
\end{figure} 

\item
Color the {\it western} and the {\it southern} borders 
of the snake graph according to the colors of the grid 
and assign the weight $q$  to the blue edges and $q^{-1}$ to the red edges.
Uncolored edges have weight~$1$; see Figure~\ref{WeightedSnake}.
\begin{figure}[H]
	\centering
	\begin{tikzpicture}[scale=0.8]
		\draw[line width=2pt,blue] (0,0)-- node[left]{$q$}(0,1);
		\draw[line width=0.7pt] (2,2)-- (2,3);
		\draw[line width=0.7pt] (0,0)-- (1,0);
		\draw[line width=2pt,blue] (1,0)-- node[below]{$q$}(2,0);
	        \draw[line width=2pt,red] (2,0)-- node[below]{$q^{-1}$}(3,0);
		\draw[line width=0.7pt] (0,1)-- (3,1);
		\draw[line width=0.7pt] (1,0)-- (1,1);
		\draw[line width=0.7pt] (2,0)-- (2,1);
		\draw[line width=2pt,red] (2,1)-- node[left]{$q^{-1}$}(2,2);
		\draw[line width=2pt,blue] (2,2)-- node[left]{$q$}(2,3);
		\draw[line width=0.7pt] (2,2)-- (3,2);
		\draw[line width=2pt,blue] (3,2)-- node[below]{$\;q$}(4,2);
		\draw[line width=2pt,red] (4,2)-- node[below]{$\;q^{-1}$}(5,2);
		\draw[line width=0.7pt] (3,0)-- (3,3);
		\draw[line width=0.7pt] (2,3)-- (5,3)--(7,3);
                 \draw[line width=0.7pt] (4,2)-- (4,3);
                 \draw[line width=2pt,red] (6,3)--  node[left]{$q^{-1}$}(6,4);
                 \draw[line width=2pt,blue] (6,4)-- node[left]{$q$}(6,5);
                 \draw[line width=0.7pt] (7,2)-- (7,5);
                 \draw[line width=0.7pt] (6,5)-- (7,5);
                 \draw[line width=0.7pt] (6,4)-- (7,4);
                 \draw[line width=0.7pt] (7,4)-- (7,5);
                 \draw[line width=0.7pt] (6,4)-- (6,5);
                 \draw[line width=2pt,blue] (5,2)-- node[below]{$q$}(6,2);
		\draw[line width=2pt,red] (6,2)-- node[below]{$q^{-1}$}(7,2);
		
		    \draw[line width=0.7pt] (5,3)-- (5,2);
		        \draw[line width=0.7pt] (6,3)-- (6,2);		
			\end{tikzpicture}
\caption{A snake graph with weighted edges induced from the grid.}
\label{WeightedSnake}
\end{figure} 
(Note that the eastern and the northern border remain uncolored and have weight~$1$.)

\end{enumerate}

\end{defn}

\begin{rem}
(a)
The above definition of weight system on snake graphs is very close to that of~\cite{EJMO},
but it does not coincide with it.
The difference between Definition~\ref{MainDefn} and that of~\cite{EJMO}
is that the first column of the $\Z^2$-grid in Figure~\ref{TheGrid} is uncolored and produces no weight on embedded snake graphs.
In~\cite{EJMO}, the first column is colored in the same way as the rest of the $\Z^2$-grid.

(b)
An important property of the coloring (or assigning weights) that we choose
in the present paper consists in the fact that
every elementary box of a snake graph has exactly one colored edge;
see Figure~\ref{SnakeBoxes}.
\begin{figure}[H]
	\centering
	\begin{tikzpicture}[scale=0.7]

	 \begin{scope}[shift={(-4,0)}]
        \draw[line width=2pt,blue] (0,2)-- (0,3);
    \draw[line width=0.7pt] (0,2)--(1,2);
    \draw[line width=0.7pt] (1,3)--(1,2);
    \draw[line width=0.7pt] (1,3)--(0,3);
    
\end{scope}

 \begin{scope}[shift={(-1,0)}]
        \draw[line width=0.7pt] (0,2)-- (0,3);
    \draw[line width=2pt,red] (0,2)--(1,2);
    \draw[line width=0.7pt] (1,3)--(1,2);
    \draw[line width=0.7pt] (1,3)--(0,3);

\end{scope}

	 \begin{scope}[shift={(2,0)}]
        \draw[line width=2pt,red] (0,2)-- (0,3);
    \draw[line width=0.7pt] (0,2)--(1,2);
    \draw[line width=0.7pt] (1,3)--(1,2);
    \draw[line width=0.7pt] (1,3)--(0,3);

\end{scope}

 \begin{scope}[shift={(5,0)}]
        \draw[line width=0.7pt] (0,2)-- (0,3);
    \draw[line width=2pt,blue] (0,2)--(1,2);
    \draw[line width=0.7pt] (1,3)--(1,2);
    \draw[line width=0.7pt] (1,3)--(0,3);

\end{scope}

	\end{tikzpicture}
\caption{Elementary boxes.}
\label{SnakeBoxes}
\end{figure}
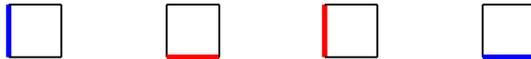 

\noindent
This property will be important for the choice of Kasteleyn orientation on snake graphs.
\end{rem}

The following definitions are standard in the theory of dimer models.

\begin{defn}
(a)
A {\it perfect matching}, or dimer cover of a graph $\G$ is a collection of edges
$\mathfrak{m}$ of $\G$, such that every vertex of $\G$
belongs to one and only one edge from $\mathfrak{m}$.

(b)
 For every perfect matching $\mathfrak{m}$ of a graph $\G$ with weighted edges,
we define the {\it weight} ${\rm{wt}}(\mathfrak{m})$ of $\mathfrak{m}$ taking the product of weights of its edges.
The weight is a power of~$q$.

(c)
Given a  graph $\mathcal{G}$ with weighted edges, 
the {\it the weighted number of perfect matchings}, or the {\it ``statistics''}
is the generating function for the weighted perfect matchings:
\begin{equation}
\label{StatEq}
\cM_q(\mathcal{G}):=
\sum_{\mathfrak{m}}
{\rm{wt}}(\mathfrak{m}).
\end{equation}
\end{defn}

\begin{rem}
Note that, according to Definition~\ref{MainDefn},
the weight of a perfect matching of a snake graph is a power of~$q$, while the generating function $\mu_{\mathcal{G}}(q)$
 is a Laurent polynomial in~$q$.
 \end{rem}

As mentioned, every rational $\frac{r}{s}\geq1$ corresponds to a snake graph $\G_{\frac{r}{s}}$
determined by the continued fraction expansion of~$\frac{r}{s}$;
see Section~\ref{DefSG}.
The main result of the paper is as follows.

\begin{thm}
\label{SnakeThm}
For every rational $\frac{r}{s}\geq1$, and the corresponding
$q$-deformed rational $\left[\frac{r}{s}\right]_q=\frac{\cR(q)}{\Sc(q)}$,  the polynomial $\cR$ in the numerator is equal (up to a scalar multiple) 
to the weighted number of perfect matchings
$\cM_q(\mathcal{G}_{\frac{r}{s}})$ in the snake graph $\mathcal{G}_{\frac{r}{s}}$ 
corresponding to~$\frac{r}{s}$.
\end{thm}

For a more detailed formulation of Theorem~\ref{SnakeThm}
and, in particular, the way of computing the scalar multiple; see Section~\ref{DetStat}.
The statement about the denominator $\Sc(q)$ of the $q$-rational $\left[\frac{r}{s}\right]_q$
is similar and uses a shorter snake graph, as in~\cite{CR,CS}.

Theorem~\ref{SnakeThm} connects the theory of $q$-rationals to that of dimers.
In the present paper, we apply this connection to give a new way to calculate $q$-rationals
in terms of Kasteleyn determinants.
We hope that this connection can give much more,
and that application of the dimer theory can produce many interesting results about $q$-rationals.

Let us mention that 
weighted perfect matchings of snake graphs (and generalized perfect matchings) were considered in~\cite{BOSZ},
but with weights that are not defined on the edges.
A preprint~\cite{CO} pursues this idea, the authors consider snake graphs
with weighted edges, but the weight assertion is different from ours.
For other weight systems on snake graphs and some particular cases,
such as the $q$-deformed Fibonacci sequence; see~\cite{OveP,CO}.

The paper is organized as follows.
In Section~\ref{RatSec}, we give a very brief introduction to the theory of $q$-rationals.
We collect three different but equivalent definitions that will be useful for this paper.
Following~\cite{CR,CS}, we explain in Section~\ref{TheSnakeSec} 
the way to associate a snake graph to every rational number.
Section~\ref{ProofSec} contains the proof of the main theorem and several examples.
In Section~\ref{ApplSec}, we apply the machinery of the dimer model system
and calculate $q$-rationals in terms of Kasteleyn determinants.
We observe a striking similarity of the Kasteleyn determinants and continuants.

\section{A brief introduction to $q$-rationals}\label{RatSec}

In this section, we recall three equivalent ways to define $q$-rationals.
All three definitions will be relevant for the sequel.

\subsection{The modular group $\PSL(2,\Z)$}

The most conceptual definition is based on the property of modular invariance.
The set $\Q\cup\{\infty\}$, identified with the rational projective line
$\Q\mathbb P^1$, admits a natural (transitive) action of
the modular group $\PSL(2,\Z)$.
For every $x\in\Q\cup\{\infty\}$, we have
\begin{equation}
\label{LFAct}
\begin{pmatrix}
a&b\\
c&d
\end{pmatrix}
\left(x\right)=
\frac{ax+b}{cx+d},
\qquad\qquad
a,b,c,d\in\Z,
\quad
ad-bc=1.
\end{equation}
Consider the pace $\Z(q)$ of rational functions with integer coefficients in one  formal variable~$q$.
This space also admits a $\PSL(2,\Z)$-action generated by two matrices
\begin{equation}
\label{RS}
R_{q}=
\begin{pmatrix}
q&1\\[2pt]
0&1
\end{pmatrix},
\qquad\qquad
L_{q}=
\begin{pmatrix}
q&0\\[2pt]
q&1
\end{pmatrix},
\end{equation}
acting by linear-fractional transformations.

There exists a unique map~\eqref{QMap}
commuting with the $\PSL_2(\Z)$-action and such that $[0]_q=0.$
This statement follows from the results of~\cite{MGOfmsigma}, 
it was explicitly formulated in~\cite{LMGadv}.
For a simple complete proof; see~\cite{MGOSur}.
In other words, the property of $\PSL(2,\Z)$-invariance can be expressed as two recurrent formulas:
\begin{equation}
\label{RecEq}
\left[x+1\right]_q=q[x]_q+1,
\qquad\qquad
\left[-\frac{1}{x}\right]_q=-\frac{1}{q[x]_q},
\end{equation}
for every~$x\in\Q\cup\{\infty\}$.
For every $x\in\Q\cup\{\infty\}$, the result of this map is denoted by
$[x]_q$ and is called the $q$-deformation of~$x$.

\begin{rem}
Note that the generators $R_{q}$ and $L_{q}$ are $q$-deformations 
of the standard generators of $\PSL(2,\Z)$
$$
T=
\begin{pmatrix}
1&1\\[2pt]
0&1
\end{pmatrix},
\qquad\qquad
L=
\begin{pmatrix}
1&0\\[2pt]
1&1
\end{pmatrix},
$$
The connection of the $\PSL(2,\Z)$-action generated by~\eqref{RS} to the
classical Burau representation of the braid group $B_3$ is discussed in~\cite{MGOV}.
\end{rem}

\begin{ex}
\label{FirstEx}
Let us give some examples that will be helpful
\begin{eqnarray*}
\left[\frac{5}{2}\right]_q&=&
\frac{1+2q+q^2+q^3}{1+q},
\\[6pt]
\left[\frac{7}{4}\right]_q&=&
\frac{1+q+2q^2+2q^3+q^4}{1+q+q^2+q^3},
\\[6pt]
\left[\frac{29}{12}\right]_q&=&
\frac{1+3q+5q^2+6q^3+6q^4+5q^5+2q^6+q^7}{1+2q+3q^2+3q^3+2q^4+q^5}.
\end{eqnarray*}
These examples can be calculated connecting the rationals to $0$
and using the recurrences~\eqref{RecEq}.
More precisely, one has 
$$
\left[\frac{5}{2}\right]_q=R_q^2L_q^2(0),
\qquad
\left[\frac{7}{4}\right]_q=R_qL_qR_q^3(0),
\qquad
\left[\frac{29}{12}\right]_q=R_q^2L_q^2R_q^2L_q^2(0).
$$
Alternatively, one can use the explicit formulas below.
\end{ex}

\subsection{$q$-deformed continued fractions}

Every rational number~$\frac{r}{s}>1$ 
has the following (regular) continued fraction expansion
\begin{equation}
\label{CFEx}
\frac{r}{s}
\quad=\quad
a_1 + \cfrac{1}{a_2 
          + \cfrac{1}{\ddots +\cfrac{1}{a_{k}} } },
\end{equation}
where the coefficients $a_i$ are positive integers such that $a_i\geq1$.
It is usually denoted by  
$$
\frac{r}{s}=\left[a_1,a_2,\ldots,a_{k-1},a_{k}\right].
$$
Note that the choice of coefficients in~\eqref{CFEx} is not unique.
The ambiguity 
\begin{equation}
\label{Ambigo}
\left[a_1,\ldots,a_{k},1\right]=
\left[a_1,\ldots,a_{k}+1\right]
\end{equation}
allows one to chose either even, or odd number $k$ of coefficients
and make the expansion~\eqref{CFEx} unique.

\begin{defn}
Given a rational~$\frac{r}{s}>1$,
the corresponding $q$-rational~$\left[\frac{r}{s}\right]_{q}$ is the rational function defined by
the following expression
\begin{equation}
\label{qa}
\left[\frac{r}{s}\right]_{q}:=
[a_1]_{q} + \cfrac{q^{a_{1}}}{[a_2]_{q^{-1}} 
          + \cfrac{q^{-a_{2}}}{[a_{3}]_{q} 
          +\cfrac{q^{a_{3}}}{[a_{4}]_{q^{-1}}
          + \cfrac{q^{-a_{4}}}{
        \cfrac{\ddots}{[a_{k-1}]_{q^{(-1)^k}}+\cfrac{q^{(-1)^ka_{k-1}}}{[a_{k}]_{q^{(-1)^{k+1}}}}}}
          } }} 
\end{equation}
where~$[a_i]_q$ are the usual $q$-integers given by~\eqref{EQInt}.
\end{defn}

Note that the result of~\eqref{qa} is independent of the choice
of the coefficients~$a_i$ in~\eqref{CFEx} (see~\cite{LMGadv}).
In particular, it stays invariant under~\eqref{Ambigo}.
Observe also the inversion of the parameter in the even terms.

\subsection{Continued fractions in the matrix form}

Recall that the classical continued fraction~\eqref{CFEx} can be presented in a matrix form.
More precisely, when $k$ is even, $k=2m$,
the rational $\frac{r}{s}$ can be recovered from the first column of the matrix
$$
\begin{pmatrix}
r&\tilde{r}\\
s&\tilde{s}
\end{pmatrix}
=R^{a_1}L^{a_2}\cdots R^{a_{2m-1}}L^{a_{2m}},
$$
When $k=2m+1$, one has
$$
\begin{pmatrix}
\tilde{r}&r\\
\tilde{s}&s
\end{pmatrix}
=R^{a_1}L^{a_2}\cdots L^{a_{2m}}R^{a_{2m+1}},
$$
In both cases $\frac{\tilde{r}}{\tilde{s}}$ is the rational with the continued fraction expansion
$$
\frac{\tilde{r}}{\tilde{s}}=
\left[a_1,a_2,\ldots,a_{k-1}\right].
$$

Similarly, is the $q$-deformed case, the rational function 
$\left[\frac{r}{s}\right]_q=\frac{\cR(q)}{\Sc(q)}$
is the quotient of the polynomials in the first (resp. second) column of the 
corresponding $q$-deformed matrix.
The formulas are slightly different in the even and odd cases:
\begin{eqnarray}
\label{Matqa}
\begin{pmatrix}
q\cR(q)&\tilde{\cR}(q)\\[2pt]
q\Sc(q)&\tilde{\Sc}(q)
\end{pmatrix}
&=&R_q^{a_1}L_q^{a_2}\cdots R_q^{a_{2m-1}}L_q^{a_{2m}},\\[6pt]
\label{MatqaOdd}
\begin{pmatrix}
q^{a_{2m+1}+1}\tilde{\cR}(q)&\cR(q)\\[2pt]
q^{a_{2m+1}+1}\tilde{\Sc}(q)&\Sc(q)
\end{pmatrix}
&=&R_q^{a_1}L_q^{a_2}\cdots L_q^{a_{2m}}R_q^{a_{2m+1}}.
\end{eqnarray}
where $\frac{\tilde{\cR}(q)}{\tilde{\Sc}(q)}$ is the $q$-analogue of~$\frac{\tilde{r}}{\tilde{s}}$
and where $\ell=a_1+a_3+\cdots+a_{2m+1}$ is the sum of odd coefficients.
The proof of~\eqref{Matqa} and~\eqref{MatqaOdd} is a straightforward computation;
see~\cite{MGOfmsigma}, Propositions~4.3 and~4.5.

\subsection{Continuants}

As mentioned, a $q$-rational is represented by an irreducible fraction 
$\left[\frac{r}{s}\right]_q=\frac{\cR(q)}{\Sc(q)}$
of polynomials with positive integer coefficients.
For simplicity, we assume that $\frac{r}{s}\geq0$.
It readily follows from~\eqref{qa} that the numerator can be expressed as the determinant
of a $3$-diagonal matrix:
\begin{equation}
\label{K+Eq}
\cR(q)=
\left|
\begin{array}{ccccccccc}
[a_1]_{q}&-1&&&\\[6pt]
q^{a_{1}}&[a_{2}]_{q^{-1}}&-1&&\\[6pt]
&q^{-a_{2}}&[a_3]_{q}&-1&&\\[6pt]
&&q^{a_{3}}&[a_{4}]_{q^{-1}}&-1&&\\[6pt]
&&&\ddots&\ddots&\!\!\ddots&\\[6pt]
&&&&q^{(-1)^{k+1}a_{k-2}}&[a_{k-1}]_{q^{(-1)^{k}}}&\!\!\!\!\!-1\\[6pt]
&&&&&q^{(-1)^{k}a_{k-1}}&\!\!\!\![a_{k}]_{q^{(-1)^{k+1}}}
\end{array}
\right|.
\end{equation}
This is a $q$-analogue of the classical Euler's {\it continuant}.
The notation for the determinant~\eqref{K+Eq} used in~\cite{MGOfmsigma}
is $\cR(q)=K(c_1,\ldots,c_{k})_{q}.$
The denominator of the $q$-rational is given by a similar formula with the sequence
of coefficients starting from $a_2$, namely
$\Sc(q)=K(c_2,\ldots,c_{k})_{q};$
see~\cite{MGOfmsigma}.

\subsection{An example: the $q$-Fibonacci sequence}\label{FiboSec}

A remarkable sequence of rationals is the quotient of consecutive 
Fibonacci numbers~$\frac{F_{n+1}}{F_{n}}=\big[\underbrace{1,1,\ldots,1}_n\big]$.
Its $q$-deformation was studied in~\cite{MGOfmsigma,LMGOV}.
The $q$-deformation
$
\left[\frac{F_{n+1}}{F_{n}}\right]_q=
\frac{\tilde\F_{n+1}(q)}{\F_n(q)}
$
 give rase to two sequences of polynomials~$\tilde\F_{n}(q)$ and~$\F_n(q)$
 in the numerator and denominator.
These polynomials are of degree~$n-2$ (for~$n\geq2$)
and are mirror of each other:
$
\tilde\F_n(q)=q^{n-2}\F_n(q^{-1}).
$
Both of them are $q$-deformations of the Fibonacci numbers.

Both of these sequences~$\tilde\F_{n}(q)$ and~$\F_n(q)$ satisfy the same recurrence
$$
\F_{n+2}(q)= [3]_q\,\F_n(q)-q^2\F_{n-2}(q),
$$
where, as before, $[3]_q=1+q+q^2$, and the initial conditions
$$
\F_0(q)=1,\quad
\F_2(q)=[2]_q
\qquad\hbox{and}\qquad
\F_1(q)=1,\quad
\F_3(q)=[3]_q
$$
(see~\cite{LMGOV}).
This allows one to calculate the sequences~$\tilde\F_{n}(q)$ and~$\F_n(q)$ recurrently.

\begin{ex}
\label{FibpPEx}
One has
$$
\begin{array}{rcl}
\left[\frac{5}{3}\right]_q&=&\displaystyle\frac{1+q+2q^2+q^3}{1+q+q^2},
\\[12pt]
\left[\frac{8}{5}\right]_q&=&\displaystyle\frac{1+2q+2q^2+2q^3+q^4}{1+2q+q^2+q^3},
\\[12pt]
\left[\frac{13}{8}\right]_q&=&\displaystyle\frac{1+2q+3q^2+3q^3+3q^4+q^5}{1+2q+2q^2+2q^3+q^4},
\\[12pt]
\left[\frac{21}{13}\right]_q&=&\displaystyle\frac{1+3q+4q^2+5q^3+4q^4+3q^5+q^6}{1+3q+3q^2+3q^3+2q^4+q^5},
\\
\ldots&\ldots&\ldots
\end{array}
$$
\end{ex}

The coefficients of the polynomials~$\tilde\F_{n}(q)$ and~$\F_n(q)$
appear in Sequences~A123245 and~A079487 of OEIS~\cite{OEIS}, respectively.

\begin{rem}
According to~\eqref{qa}, the limit of the sequence of rational functions 
$\left[\frac{F_{n+1}}{F_{n}}\right]_q$
can be calculated using the continued fraction
\begin{equation}
\label{FiboCF}
\left[\varphi\right]_q=
1 + \cfrac{q^{2}}{q
          + \cfrac{1}{1 
          +\cfrac{q^{2}}{q
          + \cfrac{1}{\ddots
       }}}} 
\end{equation}
cut after $n$ terms.
Note that the infinite continued fraction~\eqref{FiboCF}
understood as a power series in~$q$ is the $q$-analogue of the golden ratio
$\varphi=\frac{1+\sqrt{5}}{2}$, as defined in~\cite{MGOexp}.
For more details on $q$-irrationals; see~\cite{Eti,LMGadv,LMGOV,OP}.
\end{rem}

\section{Rational numbers and snake graphs}\label{TheSnakeSec}

Although snake graphs, called by different names, have always been present forever in mathematics, 
their first association with rational numbers and continued fractions is due to 
\c{C}anak\c{c}i and Schiffler; see~\cite{CR} (and also~\cite{CS}).
In this section, we briefly overview their construction and one of their results.

\subsection{Definition of the snake graph $\G_{\frac{r}{s}}$}\label{DefSG}

Following~\cite{CR}, let us associate a sequence of $+$ and $-$ signs to every snake graph.
One starts with the initial box with the south and east edges labeled by $-$ 
and the north edge labeled by $+$; see Figure~\ref{BoxSign},
\begin{figure}[H]
	\centering
	\begin{tikzpicture}[scale=1.2]

        \draw[line width=0.7pt] (0,0)-- (0,1);
    \draw[line width=0.7pt] (0,0)--(1,0);
    \draw[line width=0.7pt] (0,1)--(1,1);
    \draw[line width=0.7pt] (1,0)--(1,1);
   
  \draw[fill=white] (0.5,0) circle (0.15);
\node[cyan] at (0.5, 0)  {$-$};

  \draw[fill=white] (0.5,1) circle (0.15);
\node[cyan] at (0.5, 1)  {$+$};

  \draw[fill=white] (1,0.5) circle (0.15);
\node[cyan] at (1,0.5)  {$-$};

	\end{tikzpicture}
\caption{Two sign arrangements on the initial box.}
\label{BoxSign}
\end{figure}
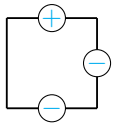 
\noindent
and then continue so that the edges of every elementary box of $\G$ are labeled 
according to the following rule.

\begin{enumerate}
\item
the north and the west edges have the same sign;
\item
the south and the east edges have the same sign;
\item
the signs on the north and south edges are opposite.
\end{enumerate}

\noindent
The sequence of signs of a snake graph starts with $-$ of the south edge of the first box
and continues with the signs of attached edges, as depicted in Figure~\ref{SignSec}.

\begin{figure}[H]
	\centering
\begin{tikzpicture}[scale=0.7]
    \begin{scope}[shift={(-6,0)}]
    \draw[line width=0.7pt] (0,0)-- (1,0);
    \draw[line width=0.7pt] (1,0)-- (2,0);
    \draw[line width=0.7pt] (0,0)-- (0,1);
    \draw[line width=0.7pt] (0,1)--(2,1)--(2,0);
    \draw[line width=0.7pt] (1,1)--(1,0);
        \draw[line width=0.7pt] (2,0)--(3,0);
    \draw[line width=0.7pt] (3,0)--(3,1)--(2,1);
    
      \draw[fill=white] (0.5,0) circle (0.2);
\node[cyan] at (0.5, 0)  {$-$};

     \draw[fill=white] (1,0.5) circle (0.2);
\node[cyan] at (1,0.5)  {$-$};

     \draw[fill=white] (2.5,1) circle (0.2);
\node[cyan] at (2.5,1)  {$+$};

     \draw[fill=white] (2,0.5) circle (0.2);
\node[cyan] at (2,0.5)  {$+$};

 \node at(1.5,-1) {$(- -,+ +)$};
   
    \end{scope}
    
  \begin{scope}[shift={(1,0)}]
    \draw[line width=0.7pt] (0,0)--(1,0);
    \draw[line width=0.7pt] (0,0)--(0,1);
    \draw[line width=0.7pt] (0,1)--(0,2);
    \draw[line width=0.7pt] (0,2)--(0,3);
    \draw[line width=0.7pt] (1,0)--(1,3);
    \draw[line width=0.7pt] (0,3)--(2,3);
    \draw[line width=0.7pt] (0,1)--(1,1);
    \draw[line width=0.7pt] (0,2)--(1,2);
    \draw[line width=0.7pt] (1,2)--(2,2);
    \draw[line width=0.7pt] (2,2)--(2,3);
    
         \draw[fill=white] (0.5,0) circle (0.2);
\node[cyan] at (0.5, 0)  {$-$};

     \draw[fill=white] (0.5,1) circle (0.2);
\node[cyan] at (0.5, 1)  {$+$};

     \draw[fill=white] (0.5,2) circle (0.2);
\node[cyan] at (0.5, 2)  {$-$};

     \draw[fill=white] (1,2.5) circle (0.2);
\node[cyan] at (1,2.5)  {$-$};

     \draw[fill=white] (2,2.5) circle (0.2);
\node[cyan] at (2,2.5)  {$+$};

    \node at(0.9,-1) {$(-,+,- -,+)$};
    \end{scope}

    \begin{scope}[shift={(6,0)}]
    \draw[line width=0.7pt] (0,0)-- (1,0);
    \draw[line width=0.7pt] (1,0)-- (2,0);
    \draw[line width=0.7pt] (0,0)-- (0,1);
    \draw[line width=0.7pt] (0,1)--(2,1)--(2,0);
    \draw[line width=0.7pt] (1,1)--(1,0);
    
            \draw[fill=white] (0.5,0) circle (0.2);
\node[cyan] at (0.5, 0)  {$-$};

     \draw[fill=white] (1,0.5) circle (0.2);
\node[cyan] at (1,0.5)  {$-$};

    \node at(2,-1) {$(- -,+ +,- -, + +)$};
    \end{scope}

    \begin{scope}[shift={(8,0)}]
    \draw[line width=0.7pt] (0,0)--(1,0);
    \draw[line width=0.7pt] (0,0)--(0,1);
    \draw[line width=0.7pt] (0,1)--(0,2);
    \draw[line width=0.7pt] (0,2)--(0,3);
    \draw[line width=0.7pt] (1,0)--(1,3);
    \draw[line width=0.7pt] (0,3)--(2,3);
    \draw[line width=0.7pt] (0,1)--(1,1);
    \draw[line width=0.7pt] (0,2)--(1,2);
    \draw[line width=0.7pt] (1,2)--(2,2);
    \draw[line width=0.7pt] (2,2)--(2,3);
     \end{scope}
     \begin{scope}[shift={(10,2)}]
    \draw[line width=0.7pt] (0,0)--(1,0);
    \draw[line width=0.7pt] (0,0)--(0,1)--(1,1)--(1,0);
    
         \draw[fill=white] (-2,-1.5) circle (0.2);
\node[cyan] at (-2,-1.5)  {$+$};
    
                \draw[fill=white] (-1.5,0) circle (0.2);
\node[cyan] at (-1.5, 0)  {$-$};

                \draw[fill=white] (-1.5,-1) circle (0.2);
\node[cyan] at (-1.5, -1)  {$+$};

     \draw[fill=white] (-1,0.5) circle (0.2);
\node[cyan] at (-1,0.5)  {$-$};

    \draw[fill=white] (0,0.5) circle (0.2);
\node[cyan] at (0,0.5)  {$+$};

    \draw[fill=white] (0.5,1) circle (0.2);
\node[cyan] at (0.5,1)  {$+$};

      \end{scope}

	\end{tikzpicture}
	\caption{}
\label{SignSec}
\end{figure} 
\noindent
For more details; see~\cite{CR,CS}.

\begin{rem}
The sequence ends up with the north, or east edge of the last box. 
The choice of the last edge is determined by the choice of even or odd
number of coefficients; see Figure~\ref{SignSec} where we choose even~$k$.
However, the choice of the last edge does not change the snake graph.
\end{rem}

Clearly, every snake graph is determined by the corresponding sequence of signs.

\begin{defn}
Given a rational $\frac{r}{s}\geq1$, 
one associates a snake graph $\G_{\frac{r}{s}}$ to $\frac{r}{s}$
in such a way that the sequence of signs of~$\G_{\frac{r}{s}}$ coincides
with the sequence of coefficients in the continued fraction
$\frac{r}{s}=[a_1,a_2,\ldots,a_{2m}]$:
$$
\Big(
\underbrace{- \cdots -}_{a_1}\,,\;
\underbrace{+ \cdots +}_{a_2}\,,\;
\underbrace{- \cdots -}_{a_3}\,,\;
\ldots\,,\;
\underbrace{+ \cdots +}_{a_{2m}}
\Big)
\qquad
\Big(
\underbrace{- \cdots -}_{a_1}\,,\;
\underbrace{+ \cdots +}_{a_2}\,,\;
\underbrace{- \cdots -}_{a_3}\,,\;
\ldots\,,\;
\underbrace{- \cdots -}_{a_{2m+1}}
\Big).
$$
\end{defn}

\begin{ex}
\label{SimpEx}
The snake graphs in Figure~\ref{SignSec} represent the rationals
(already discussed in Example~\ref{FirstEx})
$$
\frac{5}{2}=[2,2]=[2,1,1],
\qquad\qquad
\frac{7}{4}=[1,1,2,1]=[1,1,3],
\qquad\hbox{and}\qquad
\frac{29}{12}=[2,2,2,2]=[2,2,2,1,1],
$$ 
respectively.
\end{ex}

The result of~\cite{CS} states that
{\it  the numerator $r$ of a rational number $\frac{r}{s}$ is equal to
the number of perfect matchings of~$\G_{\frac{r}{s}}$}.
To calculate the denominator~$s$, one needs to consider the continued fraction with
first coefficient removed.
More precisely, take
$\frac{r'}{s'}:=[a_2,a_3,\ldots,a_{k}]$
 instead of $\frac{r}{s}=[a_1,a_2,\ldots,a_{k}]$.
The denominator $s$  is then equal to
the number of perfect matchings of~$\G_{\frac{r'}{s'}}$.

\subsection{The ladders}
To better explain the correspondence between snake graphs and rationals, 
let us describe the fragments of snake graphs 
that represent fragments of the sign sequence with constant signs.
They constitute ``ladders'',
as represented in Figure~\ref{LadFig}.
\begin{figure}[H]
	\centering
\begin{tikzpicture}[scale=0.7]
    
  \begin{scope}[shift={(-4,0)}]
           \draw[line width=0.7pt] (0,1)-- (0,2);
    \draw[line width=2pt,blue] (0,2)-- (0,3);
    \draw[line width=0.7pt] (1,1)--(1,3);
    \draw[line width=0.7pt] (0,3)--(2,3);
    \draw[line width=2pt,blue] (0,1)--(1,1);
    \draw[line width=0.7pt] (0,2)--(1,2);
    \draw[line width=2pt,blue] (1,2)-- (2,2);
    \draw[line width=0.7pt] (2,2)--(2,3);
    
     \draw[line width=2pt,blue] (-1,1)-- (-1,2);
     \draw[line width=0.7pt] (-1,2)-- (0,2);
 \draw[line width=0.7pt] (-1,1)-- (0,1);
  
 \draw[line width=0.7pt] (2,2)--(2,4);
\draw[line width=0.7pt] (2,3)--(3,3);
\draw[line width=2pt,blue] (1,3)-- (1,4);
 \draw[line width=0.7pt] (1,4)--(3,4);
  \draw[line width=2pt,blue] (2,3)-- (3,3);
  \draw[line width=0.7pt] (3,3)--(3,4);

\draw[line width=0.7pt] (3,4)--(3,5);
\draw[line width=2pt,blue] (2,4)--(2,5);
\draw[line width=0.7pt] (2,5)--(3,5);

      \draw[fill=white] (-0.5,1) circle (0.2);
\node[cyan] at (-0.5, 1)  {$-$};

     \draw[fill=white] (0,1.5) circle (0.2);
\node[cyan] at (0,1.5)  {$-$};

      \draw[fill=white] (0.5,2) circle (0.2);
\node[cyan] at (0.5, 2)  {$-$};

    \draw[fill=white] (1,2.5) circle (0.2);
\node[cyan] at (1,2.5)  {$-$};

      \draw[fill=white] (1.5,3) circle (0.2);
\node[cyan] at (1.5, 3)  {$-$};

    \draw[fill=white] (2,3.5) circle (0.2);
\node[cyan] at (2,3.5)  {$-$};

      \draw[fill=white] (2.5,4) circle (0.2);
\node[cyan] at (2.5, 4)  {$-$};

    \draw[fill=white] (3,4.5) circle (0.2);
\node[cyan] at (3,4.5)  {$-$};

    \end{scope}

  \begin{scope}[shift={(4,0)}]
           \draw[line width=0.7pt] (0,1)-- (0,2);
    \draw[line width=2pt,blue] (0,2)-- (0,3);
    \draw[line width=0.7pt] (1,1)--(1,3);
    \draw[line width=0.7pt] (0,3)--(2,3);
    \draw[line width=2pt,blue] (0,1)--(1,1);
    \draw[line width=0.7pt] (0,2)--(1,2);
    \draw[line width=2pt,blue] (1,2)-- (2,2);
    \draw[line width=0.7pt] (2,2)--(2,3);
    
     \draw[line width=2pt,blue] (-1,1)-- (-1,2);
     \draw[line width=0.7pt] (-1,2)-- (0,2);
 \draw[line width=0.7pt] (-1,1)-- (0,1);
  
 \draw[line width=0.7pt] (2,2)--(2,3);
 \draw[line width=2pt,red] (2,2)-- (3,2);
   \draw[line width=0.7pt] (3,2)--(3,4);
\draw[line width=0.7pt] (2,3)--(3,3);
\draw[line width=2pt,red] (2,3)-- (2,4);
 \draw[line width=0.7pt] (2,4)--(4,4);
  \draw[line width=2pt,red] (3,3)-- (4,3);
  \draw[line width=0.7pt] (4,3)--(4,4);

\draw[line width=0.7pt] (4,4)--(4,5);
\draw[line width=0.7pt] (4,4)--(4,5);
\draw[line width=2pt,red] (3,4)--(3,5);
\draw[line width=0.7pt] (3,5)--(4,5);
\draw[line width=2pt,red] (4,4)--(5,4);
\draw[line width=0.7pt] (4,5)--(5,5);
\draw[line width=0.7pt] (5,4)--(5,5);

     \draw[fill=white] (-0.5,1) circle (0.2);
\node[cyan] at (-0.5, 1)  {$-$};

     \draw[fill=white] (0,1.5) circle (0.2);
\node[cyan] at (0,1.5)  {$-$};

      \draw[fill=white] (0.5,2) circle (0.2);
\node[cyan] at (0.5, 2)  {$-$};

    \draw[fill=white] (1,2.5) circle (0.2);
\node[cyan] at (1,2.5)  {$-$};

    \draw[fill=white] (2,2.5) circle (0.2);
\node[cyan] at (2,2.5)  {$+$};

     \draw[fill=white] (2.5,3) circle (0.2);
\node[cyan] at (2.5, 3)  {$+$};

    \draw[fill=white] (3,3.5) circle (0.2);
\node[cyan] at (3,3.5)  {$+$};

     \draw[fill=white] (3.5,4) circle (0.2);
\node[cyan] at (3.5, 4)  {$+$};

    \draw[fill=white] (4,4.5) circle (0.2);
\node[cyan] at (4,4.5)  {$+$};

     \draw[fill=white] (4.5,5) circle (0.2);
\node[cyan] at (4.5, 5)  {$+$};

    \end{scope}
       
	\end{tikzpicture}
\caption{\hskip 0.2cm a) the snake for $8$; \hskip 1cm b) attaching two ladders.}
\label{LadFig}
\end{figure}
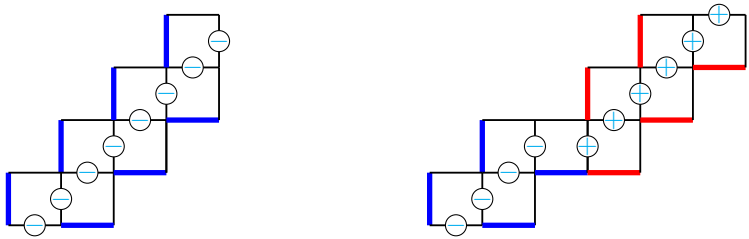 

\noindent
Each ladder has the same coloring, 
all of its boxes have a blue edge, or a red edge.

\section{Proof of the main theorem}\label{ProofSec}

In this section we prove (a slightly reformulated version of) Theorem~\ref{SnakeThm}
stating that $q$-rational numbers count perfect matchings of
snake graphs with weighted edges.

\subsection{Detailed statement of the main theorem}\label{DetStat}

Consider a rational $\frac{r}{s}\geq1$, and chose its continued fraction expansion with even numbers of terms:
$$
\frac{r}{s}=[a_1,a_2,\ldots,a_{2m}].
$$
This can always be done thanks to~\eqref{Ambigo}.
Define the integer $n_{\frac{r}{s}}$ as the sum of the even coefficients of the continued fraction decreased by~$1$:
\begin{equation}
\label{Len}
n=n_{\frac{r}{s}}:=a_2+a_4+\cdots+a_{2m}-1.
\end{equation}

The following statement is a more detailed reformulation of our main result
announced in the introduction as Theorem~\ref{SnakeThm}.

\begin{thm}
\label{SnakeThmBis}
For every $q$-rational $\left[\frac{r}{s}\right]_q=\frac{\cR(q)}{\Sc(q)}$,
such that $\frac{r}{s}\geq1$, the polynomial $\cR$ in the numerator is given by
\begin{equation}
\label{MainEq}
\cR(q)=q^{n}\cM_q(\G_{\frac{r}{s}}),
\end{equation}
where $\G_{\frac{r}{s}}$ is the snake graph corresponding to~$\frac{r}{s}$,
and where $n$ is as in~\eqref{Len}.
\end{thm}

Note that we need the continued fraction with even number of coefficients 
to define the coefficient~$n$, but
paradoxically, the proof is not based on the even continued fraction.
Let $\frac{r}{s}\geq1$ be a rational, consider its continued fraction expansion
$$
\frac{r}{s}=[a_1,a_2,\ldots,a_{k-1},a_{k}]
$$
in which, using the ambiguity~\eqref{Ambigo}, we can assume that $a_k\geq2$.
The proof of Theorem~\ref{SnakeThmBis}
will use the induction on the sum of coefficients
$$
N=a_1+a_2+\cdots+a_{k-1}+a_{k},
$$
or equivalently on the number of boxes in the snake graph.

\subsection{Two inductive cases}

Suppose that a rational $\frac{r'}{s'}\geq1$ has the sum of coefficients $N'=N-1$.
By induction assumption we assume that the statement of Theorem~\ref{SnakeThmBis}
holds for~$\frac{r'}{s'}\geq1$.

There are two different cases.
\begin{enumerate}
\item
The continued fraction expansion of $\frac{r'}{s'}$ is
\begin{equation}
\label{OneEq}
\frac{r'}{s'}=\left[a_1,a_2,\ldots,a_{2m-1},\,a_{2m}-1\right],
\end{equation}
when $k=2m$.

\item
The continued fraction expansion of $\frac{r'}{s'}$  is 
\begin{equation}
\label{TwoEq}
\frac{r'}{s'}=\left[a_1,a_2,\ldots,a_{2m+1}-1\right],
\end{equation}
when $k=2m+1$.

\end{enumerate}

We need to consider each of these cases separately.

\subsection{Case $1$}

The snake graph $\G_{\frac{r}{s}}$ is obtained from $\G_{\frac{r'}{s'}}$ by adding a box,
as illustrated in Figure~\ref{SnakeProof1}.
\begin{figure}[H]
	\centering
\begin{tikzpicture}[scale=0.6]
    
  \begin{scope}[shift={(-4,0)}]
           \draw[line width=0.7pt] (0,1)-- (0,2);
    \draw[line width=2pt,blue] (0,2)-- (0,3);
    \draw[line width=0.7pt] (1,1)--(1,3);
    \draw[line width=0.7pt] (0,3)--(2,3);
    \draw[line width=2pt,blue] (0,1)--(1,1);
    \draw[line width=0.7pt] (0,2)--(1,2);
    \draw[line width=2pt,blue] (1,2)-- (2,2);
    \draw[line width=0.7pt] (2,2)--(2,3);
    
     \draw[line width=2pt,blue] (-1,1)-- (-1,2);
     \draw[line width=0.7pt] (-1,2)-- (0,2);
 \draw[line width=0.7pt] (-1,1)-- (0,1);
     \draw[dashed] (-1,1)--(-1,0);
 \draw[dashed] (0,1)--(0,0);
 
 \draw[line width=0.7pt] (2,2)--(2,3);
 \draw[line width=2pt,red] (2,2)-- (3,2);
   \draw[line width=0.7pt] (3,2)--(3,4);
\draw[line width=0.7pt] (2,3)--(3,3);
\draw[line width=2pt,red] (2,3)-- (2,4);
 \draw[line width=0.7pt] (2,4)--(4,4);
  \draw[line width=2pt,red] (3,3)-- (4,3);
  \draw[line width=0.7pt] (4,3)--(4,4);

\draw[line width=0.7pt] (4,4.3)--(4,5.3);
\draw[line width=0.7pt] (3,4.3)--(4,4.3);
\draw[line width=2pt,red] (3,4.3)--(3,5.3);
\draw[line width=0.7pt] (3,5.3)--(4,5.3);

    \end{scope}

  \begin{scope}[shift={(4,0)}]
           \draw[line width=0.7pt] (0,1)-- (0,2);
    \draw[line width=2pt,blue] (0,2)-- (0,3);
    \draw[line width=0.7pt] (1,1)--(1,3);
    \draw[line width=0.7pt] (0,3)--(2,3);
    \draw[line width=2pt,blue] (0,1)--(1,1);
    \draw[line width=0.7pt] (0,2)--(1,2);
    \draw[line width=2pt,blue] (1,2)-- (2,2);
    \draw[line width=0.7pt] (2,2)--(2,3);
    
     \draw[line width=2pt,blue] (-1,1)-- (-1,2);
     \draw[line width=0.7pt] (-1,2)-- (0,2);
 \draw[line width=0.7pt] (-1,1)-- (0,1);
     \draw[dashed] (-1,1)--(-1,0);
 \draw[dashed] (0,1)--(0,0);
 
 \draw[line width=0.7pt] (2,2)--(2,3);
 \draw[line width=2pt,red] (2,2)-- (3,2);
   \draw[line width=0.7pt] (3,2)--(3,4);
\draw[line width=0.7pt] (2,3)--(3,3);
\draw[line width=2pt,red] (2,3)-- (2,4);
 \draw[line width=0.7pt] (2,4)--(4,4);
  \draw[line width=2pt,red] (3,3)-- (4,3);
  \draw[line width=0.7pt] (4,3)--(4,4);

\draw[line width=0.7pt] (4,4)--(4,5);
\draw[line width=0.7pt] (4.3,4)--(4.3,5);
\draw[line width=2pt,red] (3,4)--(3,5);
\draw[line width=0.7pt] (3,5)--(4,5);
\draw[line width=2pt,red] (4.3,4)--(5.3,4);
\draw[line width=0.7pt] (4.3,5)--(5.3,5);
\draw[line width=0.7pt] (5.3,4)--(5.3,5);

    \end{scope}

	\end{tikzpicture}
\caption{Attachment of a box to $\G_{\frac{r'}{s'}}$: Case $1$.}
\label{SnakeProof1}
\end{figure}
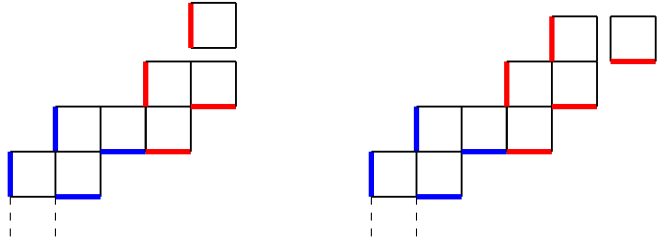 

\begin{lem}
\label{Case1Lem}
The weighted number of perfect matchings of $\G_{\frac{r}{s}}$ is given by
\begin{equation}
\label{AddEq}
\cM_q(\G_{\frac{r}{s}})=
\cM_q(\G_{\frac{r'}{s'}})+q^{-a_{2m}}\cM_q(\G_{\frac{\tilde{r}}{\tilde{s}}}),
\end{equation}
where $\frac{r'}{s'}$ is as in~\eqref{OneEq},
and where $\frac{\tilde{r}}{\tilde{s}}$ is the rational that has the continued fraction expansion
$$
\frac{\tilde{r}}{\tilde{s}}=[a_1,a_2,\ldots,a_{2m-1}-1].
$$
\end{lem}

\begin{proof}
The graph $\G_{\frac{r}{s}}$ has one edge (the exterior edge of the attached box)
that has no common vertices with $\G_{\frac{r'}{s'}}$.
The perfect matchings of $\G_{\frac{r}{s}}$ containing this edge are identified with
perfect matchings of $\G_{\frac{r'}{s'}}$, 
the number of such perfect matchings is $\cM_q(\G_{\frac{r'}{s'}})$.
This constitutes the first summand in~\eqref{AddEq}.
Similarly, the number of prefect matchings of $\G_{\frac{r}{s}}$ 
that do not contain the edge exterior to $\G_{\frac{r'}{s'}}$ is equal to
$q^{-a_{2m}}\cM_q(\G_{\frac{\tilde{r}}{\tilde{s}}})$.
Indeed, such a perfect matching is completely determined on the last ladder of $\G_{\frac{r}{s}}$.

Hence the lemma.
\end{proof}

Let us deduce Theorem~\ref{SnakeThmBis}, Case~1 from Lemma~\ref{Case1Lem}.
By~\eqref{Matqa}, the matrix of the $q$-deformed continued fraction of $\left[\frac{r}{s}\right]_q$ is
connected with that of $\left[\frac{r'}{s'}\right]_q$ in the following way
$$
\begin{pmatrix}
q\cR(q)&\tilde{\cR}(q)\\[2pt]
q\Sc(q)&\tilde{\Sc}(q)
\end{pmatrix}
=
\begin{pmatrix}
q\cR'(q)&\tilde{\cR}(q)\\[2pt]
q\Sc'(q)&\tilde{\Sc}(q)
\end{pmatrix}L_q,
$$
where $L_q$ is as in~\eqref{RS}.
Therefore, the numerator of the $q$-rational $\left[\frac{r}{s}\right]_q$ is given by
$$
\cR(q)=q\cR'(q)+\tilde{\cR}(q).
$$
By induction assumption, $\cR'(q)$ and $\tilde{\cR}(q)$ are given by Theorem~\ref{SnakeThmBis}.
More precisely,
\begin{eqnarray*}
\cR'(q)&=&q^{a_2+a_4+\cdots+a_{2m}-2}\cM_q(\G_{\frac{r'}{s'}}),\\
\tilde{\cR}(q)&=&q^{a_2+a_4+\cdots+a_{2m-2}-1}\cM_q(\G_{\frac{\tilde{r}}{\tilde{s}}}).
\end{eqnarray*}
We finally obtain
$$
\cR(q)=q^n\cM_q(\G_{\frac{r'}{s'}})+q^{n-a_{2m}}\cM_q(\G_{\frac{\tilde{r}}{\tilde{s}}}),
$$
where $n=a_2+a_4+\cdots+a_{2m}-1$.
Lemma~\ref{Case1Lem} then implies that $\cR(q)=q^n\cM_q(\G_{\frac{r}{s}})$
as stated in Theorem~\ref{SnakeThmBis}.

\subsection{Case $2$}

The snake graph $\G_{\frac{r}{s}}$ is obtained from $\G_{\frac{r'}{s'}}$ by adding a box,
as illustrated in Figure~\ref{SnakeProof2}.
\begin{figure}[H]
	\centering
\begin{tikzpicture}[scale=0.6]
    
  \begin{scope}[shift={(-4,0)}]
           \draw[line width=0.7pt] (0,1)-- (0,2);
    \draw[line width=2pt,red] (0,2)-- (0,3);
    \draw[line width=0.7pt] (1,1)--(1,3);
    \draw[line width=0.7pt] (0,3)--(2,3);
    \draw[line width=2pt,red] (0,1)--(1,1);
    \draw[line width=0.7pt] (0,2)--(1,2);
    \draw[line width=2pt,red] (1,2)-- (2,2);
    \draw[line width=0.7pt] (2,2)--(2,3);
    
     \draw[line width=2pt,red] (-1,1)-- (-1,2);
     \draw[line width=0.7pt] (-1,2)-- (0,2);
 \draw[line width=0.7pt] (-1,1)-- (0,1);
     \draw[dashed] (-1,1)--(-1,0);
 \draw[dashed] (0,1)--(0,0);
 
 \draw[line width=0.7pt] (2,2)--(2,3);
 \draw[line width=2pt,blue] (2,2)-- (3,2);
   \draw[line width=0.7pt] (3,2)--(3,4);
\draw[line width=0.7pt] (2,3)--(3,3);
\draw[line width=2pt,blue] (2,3)-- (2,4);
 \draw[line width=0.7pt] (2,4)--(4,4);
  \draw[line width=2pt,blue] (3,3)-- (4,3);
  \draw[line width=0.7pt] (4,3)--(4,4);

\draw[line width=0.7pt] (4,4.3)--(4,5.3);
\draw[line width=0.7pt] (3,4.3)--(4,4.3);
\draw[line width=2pt,blue] (3,4.3)--(3,5.3);
\draw[line width=0.7pt] (3,5.3)--(4,5.3);

    \end{scope}

  \begin{scope}[shift={(4,0)}]
           \draw[line width=0.7pt] (0,1)-- (0,2);
    \draw[line width=2pt,red] (0,2)-- (0,3);
    \draw[line width=0.7pt] (1,1)--(1,3);
    \draw[line width=0.7pt] (0,3)--(2,3);
    \draw[line width=2pt,red] (0,1)--(1,1);
    \draw[line width=0.7pt] (0,2)--(1,2);
    \draw[line width=2pt,red] (1,2)-- (2,2);
    \draw[line width=0.7pt] (2,2)--(2,3);
    
     \draw[line width=2pt,red] (-1,1)-- (-1,2);
     \draw[line width=0.7pt] (-1,2)-- (0,2);
 \draw[line width=0.7pt] (-1,1)-- (0,1);
     \draw[dashed] (-1,1)--(-1,0);
 \draw[dashed] (0,1)--(0,0);
 
 \draw[line width=0.7pt] (2,2)--(2,3);
 \draw[line width=2pt,blue] (2,2)-- (3,2);
   \draw[line width=0.7pt] (3,2)--(3,4);
\draw[line width=0.7pt] (2,3)--(3,3);
\draw[line width=2pt,blue] (2,3)-- (2,4);
 \draw[line width=0.7pt] (2,4)--(4,4);
  \draw[line width=2pt,blue] (3,3)-- (4,3);
  \draw[line width=0.7pt] (4,3)--(4,4);

\draw[line width=0.7pt] (4,4)--(4,5);
\draw[line width=0.7pt] (4.3,4)--(4.3,5);
\draw[line width=2pt,blue] (3,4)--(3,5);
\draw[line width=0.7pt] (3,5)--(4,5);
\draw[line width=2pt,blue] (4.3,4)--(5.3,4);
\draw[line width=0.7pt] (4.3,5)--(5.3,5);
\draw[line width=0.7pt] (5.3,4)--(5.3,5);

    \end{scope}

	\end{tikzpicture}
\caption{Attachment of a box to $\G_{\frac{r'}{s'}}$: Case $2$.}
\label{SnakeProof2}
\end{figure}
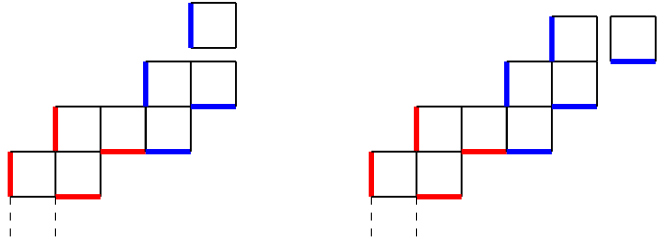 

\noindent
In a similar with the Case 1 way we obtain for the perfect matchings
$$
\cM_q(\G_{\frac{r}{s}})=
\cM_q(\G_{\frac{r'}{s'}})+q^{a_{2m+1}}\cM_q(\G_{\frac{\tilde{r}}{\tilde{s}}}).
$$

The matrix of the $q$-deformed continued fraction of $\left[\frac{r}{s}\right]_q$
in this case is as follows
$$
\begin{pmatrix}
q^{a_{2m+1}+1}\tilde{\cR}(q)&\cR(q)\\[2pt]
q^{a_{2m+1}+1}\tilde{\Sc}(q)&\Sc(q)
\end{pmatrix}
=
\begin{pmatrix}
q^{a_{2m+1}}\tilde{\cR}(q)&\cR'(q)\\[2pt]
q^{a_{2m+1}}\tilde{\Sc}(q)&\Sc'(q)
\end{pmatrix}R_q,
$$
from where we deduce
$$
\cR(q)=\cR'(q)+q^{a_{2m+1}}\tilde{R}(q).
$$
One obtains the desired result comparing this with the above formula for perfect matchingss.

Theorem~\ref{SnakeThmBis} is proved.

\subsection{A few examples}\label{ExSec}

For the simple examples from Figure~\ref{SignSec}, we have the following weighted snake graphs.
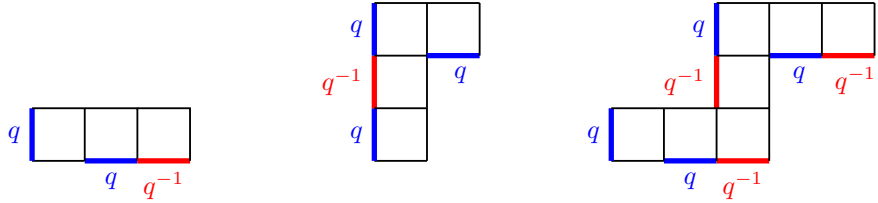
\begin{figure}[H]
	\centering
\begin{tikzpicture}[scale=0.7]
    \begin{scope}[shift={(-6,0)}]
    \draw[line width=0.7pt] (0,0)-- (1,0);
    \draw[line width=2pt,blue] (1,0)-- node[below]{$q$}(2,0);
    \draw[line width=2pt,blue] (0,0)-- node[left]{$q$}(0,1);
    \draw[line width=0.7pt] (0,1)--(2,1)--(2,0);
    \draw[line width=0.7pt] (1,1)--(1,0);
     \draw[line width=2pt,red] (2,0)-- node[below]{$q^{-1}$}(3,0);
    \draw[line width=0.7pt] (3,0)--(3,1)--(2,1);
    \end{scope}
    
  \begin{scope}[shift={(0.5,0)}]
    \draw[line width=0.7pt] (0,0)--(1,0);
    \draw[line width=2pt,blue] (0,0)-- node[left]{$q$}(0,1);
    \draw[line width=2pt,red] (0,1)--node[left]{$q^{-1}$}(0,2);
    \draw[line width=2pt,blue] (0,2)--node[left]{$q$}(0,3);
    \draw[line width=0.7pt] (1,0)--(1,3);
    \draw[line width=0.7pt] (0,3)--(2,3);
    \draw[line width=0.7pt] (0,1)--(1,1);
    \draw[line width=0.7pt] (0,2)--(1,2);
    \draw[line width=2pt,blue] (1,2)--node[below]{$\;\;q$}(2,2);
    \draw[line width=0.7pt] (2,2)--(2,3);
    \end{scope}

    \begin{scope}[shift={(5,0)}]
    \draw[line width=0.7pt] (0,0)-- (1,0);
    \draw[line width=2pt,blue] (1,0)-- node[below]{$q$}(2,0);
     \draw[line width=2pt,red] (2,0)-- node[below]{$q^{-1}$}(3,0);
    \draw[line width=2pt,blue] (0,0)-- node[left]{$q$}(0,1);
    \draw[line width=2pt,blue] (3,2)-- node[below]{$\;q$}(4,2);
     \draw[line width=2pt,red] (4,2)-- node[below]{$\;\;q^{-1}$}(5,2);
     \draw[line width=0.7pt] (1,0)-- (1,1);
    \draw[line width=0.7pt] (2,0)-- (2,1);
    \draw[line width=0.7pt] (3,0)-- (3,3);
    \draw[line width=2pt,red] (2,1)-- node[left]{$q^{-1}$}(2,2);
 \draw[line width=2pt,blue] (2,2)-- node[left]{$q$}(2,3);
    \draw[line width=0.7pt] (4,2)-- (4,3);
     \draw[line width=0.7pt] (5,2)-- (5,3);
      \draw[line width=0.7pt] (2,3)-- (5,3);
       \draw[line width=0.7pt] (2,2)-- (3,2);
       \draw[line width=0.7pt] (0,1)-- (3,1);
    \end{scope}
    
	\end{tikzpicture}
\caption{Colored snake graphs corresponding to~$\left[\frac{5}{2}\right]_q,\left[\frac{7}{4}\right]_q,$ and $\left[\frac{29}{12}\right]_q$.}
\label{SnakeColorEx}
\end{figure} 

One then checks that the corresponding weighted numbers of perfect matchings are as follows
\begin{eqnarray*}
q^{-1}+2+q+q^2 &=& q^{-1}\,\cR_{\frac{5}{2}}(q),\\
q^{-1}+1+2q+2q^2+q^3 &=& q^{-1}\,\cR_{\frac{7}{4}}(q),\\
q^{-3}+2q^{-2}+5q^{-1}+6+6q+5q^2+2q^3+q^4 &=& q^{-3}\,\cR_{\frac{29}{12}}(q),
\end{eqnarray*}
in full accordance with Theorem~\ref{SnakeThmBis}.

More examples will be given in Section~\ref{FiKast}.

\section{Applications: Kasteleyn determinants}\label{ApplSec}

As an application of Theorem~\ref{SnakeThm},
we are able to use the standard technique of dimer model theory.
In particular, we calculate the polynomials $\cR(q)$ and $\Sc(q)$ 
in terms of the determinants of Kasteleyn matrices.

\subsection{A brief summary on Kasteleyn determinants}

We recall here a basic fact from the dimer theory,
under the assumption that $\G$ is an oriented bipartite graph 
that has the same number of black and white vertices,
and the set $\G_1$ of simple (one valent) edges.

\begin{defn}

(a)
A graph is called {\it bipartite} if its vertices are colored in black and white and every edge has endpoints of opposite color.

(b)
Given a bipartite graph, it said to be equipped with a {\it Kasteleyn orientation} if 

\begin{enumerate}

\item
for every cycle of length $4m$ in the graph the number of arrows
from black vertices to white vertices is odd.

\item

for every cycle of length $4m+2$ in the graph the number of arrows
from black vertices to white vertices is even.

\end{enumerate}
\end{defn}

Let $b_1,b_2,\ldots,b_\ell$ and $w_1,w_2,\ldots,w_\ell$ be black and white vertices, respectively.
The corresponding Kasteleyn is the $\ell\times\ell$ matrix $M(\G)$ such that
$$
\begin{array}{rcl}
M_{ij} &=&
\left\{ 
\begin{array}{ll}
{\rm{wt}}(b_i,w_j), & (b_i,w_j)\in\G_1,\\[2pt]
-{\rm{wt}}(b_i,w_j), & (w_i,b_j)\in\G_1,\\[2pt]
0, & \hbox{$w_i$ and $b_j$ are disconnected}.
\end{array}
\right.
\end{array}
$$
The celebrated theorem, initially due to Kasteleyn~\cite{Kas} and strengthened by several authors,
states that the (modulus of the) determinant of the Kasteleyn matrix 
equals the weighted number of perfect matchings:
$$
\left|\det(M(\G))\right|=\cM(\G).
$$

\subsection{Choice of Kasteleyn orientation on snake graphs}

Every snake graph can be viewed as a bipartite graph in an obvious way; 
see Figures~\ref{BoxBip} and~\ref{SnakeBip}.
To make the choice unique, we assume that the down-left vertex is black.

Recall that every elementary box of a snake graph has exactly one colored edge
(as in Figure~\ref{SnakeBoxes}).
We will choose the following Kasteleyn orientation of snake graphs.

\begin{defn}
\label{KastOr}
(a)
Colored edges are oriented from black to white;

(b) uncolored edges are oriented from white to black.
\end{defn}

\noindent
With this convention the conditions of Kasteleyn orientation are clearly preserved.

There are exactly four different types of elementary boxes:

\begin{figure}[H]
	\centering
	\begin{tikzpicture}[scale=1]

	 \begin{scope}[shift={(-4,0)}]
        \draw[line width=2pt,blue] (0,2)-- (0,3);
    \draw[line width=0.7pt] (0,2)--(1,2);
    \draw[line width=0.7pt] (1,3)--(1,2);
    \draw[line width=0.7pt] (1,3)--(0,3);
    
     \draw[-{Stealth[length=2.5mm]}] (1,0)-- (0.07,0);
 \draw[-{Stealth[length=2.5mm]}] (0,0)-- (0,0.93);
\draw[-{Stealth[length=2.5mm]}] (0,1)--(0.93,1);
\draw[-{Stealth[length=2.5mm]}] (1,0)--(1,0.93);

    \draw[fill=black] (0,0) circle (0.08);
\draw[fill=white] (1,0) circle (0.08);
\draw[fill=black] (1,1) circle (0.08);
\draw[fill=white] (0,1) circle (0.08);

  \draw[fill=black] (0,2) circle (0.08);
\draw[fill=white] (1,2) circle (0.08);
\draw[fill=black] (1,3) circle (0.08);
\draw[fill=white] (0,3) circle (0.08);

\end{scope}

 \begin{scope}[shift={(2,0)}]
        \draw[line width=0.7pt] (0,2)-- (0,3);
    \draw[line width=2pt,red] (0,2)--(1,2);
    \draw[line width=0.7pt] (1,3)--(1,2);
    \draw[line width=0.7pt] (1,3)--(0,3);
    
   \draw[-{Stealth[length=2.5mm]}] (0,0)--(0.93,0);
\draw[-{Stealth[length=2.5mm]}] (0,1)--(0,0.07);
\draw[-{Stealth[length=2.5mm]}] (0,1)--(0.93,1);
\draw[-{Stealth[length=2.5mm]}] (1,0)--(1,0.93);

    \draw[fill=black] (0,0) circle (0.08);
\draw[fill=white] (1,0) circle (0.08);
\draw[fill=black] (1,1) circle (0.08);
\draw[fill=white] (0,1) circle (0.08);

  \draw[fill=black] (0,2) circle (0.08);
\draw[fill=white] (1,2) circle (0.08);
\draw[fill=black] (1,3) circle (0.08);
\draw[fill=white] (0,3) circle (0.08);

\end{scope}

	 \begin{scope}[shift={(-1,0)}]
        \draw[line width=2pt,red] (0,2)-- (0,3);
    \draw[line width=0.7pt] (0,2)--(1,2);
    \draw[line width=0.7pt] (1,3)--(1,2);
    \draw[line width=0.7pt] (1,3)--(0,3);
    
\draw[-{Stealth[length=2.5mm]}] (0,1)--(0,0.07);
\draw[-{Stealth[length=2.5mm]}] (1,1)--(1,0.07);
  \draw[-{Stealth[length=2.5mm]}] (1,1)-- (0.07,1);
   \draw[-{Stealth[length=2.5mm]}] (0,0)--(0.93,0);

    \draw[fill=white] (0,0) circle (0.08);
\draw[fill=black] (1,0) circle (0.08);
\draw[fill=white] (1,1) circle (0.08);
\draw[fill=black] (0,1) circle (0.08);

  \draw[fill=white] (0,2) circle (0.08);
\draw[fill=black] (1,2) circle (0.08);
\draw[fill=white] (1,3) circle (0.08);
\draw[fill=black] (0,3) circle (0.08);

\end{scope}

 \begin{scope}[shift={(5,0)}]
        \draw[line width=0.7pt] (0,2)-- (0,3);
    \draw[line width=2pt,blue] (0,2)--(1,2);
    \draw[line width=0.7pt] (1,3)--(1,2);
    \draw[line width=0.7pt] (1,3)--(0,3);
    
\draw[-{Stealth[length=2.5mm]}] (0,0)-- (0,0.93);
\draw[-{Stealth[length=2.5mm]}] (1,1)--(1,0.07);
  \draw[-{Stealth[length=2.5mm]}] (1,1)-- (0.07,1);
   \draw[-{Stealth[length=2.5mm]}] (1,0)-- (0.07,0);

    \draw[fill=white] (0,0) circle (0.08);
\draw[fill=black] (1,0) circle (0.08);
\draw[fill=white] (1,1) circle (0.08);
\draw[fill=black] (0,1) circle (0.08);

  \draw[fill=white] (0,2) circle (0.08);
\draw[fill=black] (1,2) circle (0.08);
\draw[fill=white] (1,3) circle (0.08);
\draw[fill=black] (0,3) circle (0.08);

\end{scope}

	\end{tikzpicture}
\caption{Kasteleyn orientation of elementary boxes.}
\label{BoxBip}
\end{figure}
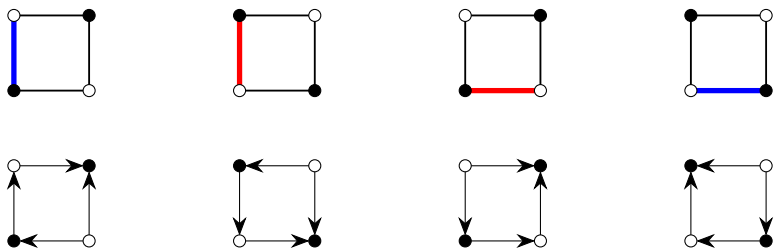 

\begin{rem}
Observe that the boxes in Figure~\ref{BoxBip} can be glued to each other if an only if the edges have the same orientation and 
the same vertex coloring.
\end{rem}

Although there are many different ways to chose a Kasteleyn orientation on every snake graph,
the orientation described above is particularly simple to define and will be considered as canonical.

\begin{ex}
The snake graph in Figure~\ref{SnakeBip} represents the rational $\frac{13}{3}=[4,3]$.

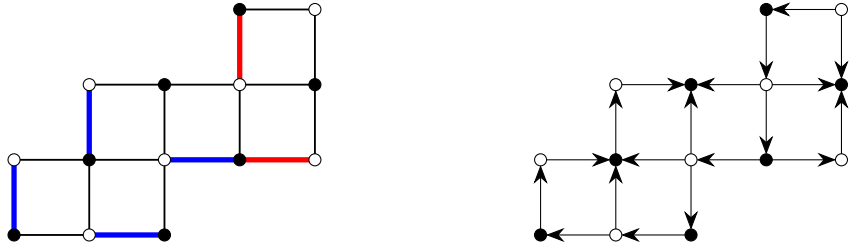
\begin{figure}[H]
	\centering
	\begin{tikzpicture}[scale=1]

	 \begin{scope}[shift={(-4,0)}]
    \draw[line width=0.7pt] (0,0)-- (1,0);
    \draw[line width=2pt,blue] (1,0)-- (2,0);
    \draw[line width=2pt,blue] (0,0)-- (0,1);
    \draw[line width=0.7pt] (0,1)--(2,1)--(2,0);
    \draw[line width=0.7pt] (1,1)--(1,0);
     \draw[line width=2pt,red] (3,1)-- (4,1);
    \draw[line width=0.7pt] (4,1)--(4,2);
 \draw[line width=0.7pt] (3,2)--(4,2);
\draw[line width=0.7pt] (3,3)--(4,3);
\draw[line width=0.7pt] (4,2)--(4,3);
 \draw[line width=2pt,red] (3,2)-- (3,3);
     \draw[line width=2pt,blue] (1,1)--(1,2);
    \draw[line width=0.7pt] (2,1)--(2,2);
    \draw[line width=0.7pt] (1,2)--(3,2);
     \draw[line width=2pt,blue] (2,1)-- (3,1);
    \draw[line width=0.7pt] (3,1)--(3,2);
    \draw[fill=black] (0,0) circle (0.08);
\draw[fill=white] (1,0) circle (0.08);
\draw[fill=black] (1,1) circle (0.08);
\draw[fill=white] (0,1) circle (0.08);
\draw[fill=black] (2,0) circle (0.08);
\draw[fill=white] (2,1) circle (0.08);
\draw[fill=black] (3,1) circle (0.08);
\draw[fill=white] (4,1) circle (0.08);
\draw[fill=black] (4,2) circle (0.08);
\draw[fill=black] (3,3) circle (0.08);
\draw[fill=white] (4,3) circle (0.08);

\draw[fill=white] (1,2) circle (0.08);
\draw[fill=black] (2,2) circle (0.08);
\draw[fill=white] (3,2) circle (0.08);
\end{scope}

	 \begin{scope}[shift={(3,0)}]
    \draw[-{Stealth[length=2.5mm]}] (1,0)-- (0.07,0);
    \draw[-{Stealth[length=2.5mm]}] (2,0)-- (1.07,0);
    \draw[-{Stealth[length=2.5mm]}] (0,0)-- (0,0.93);
    \draw[-{Stealth[length=2.5mm]}] (2,1)--(2,0.07);
    \draw[-{Stealth[length=2.5mm]}] (1,0)--(1,0.93);
   \draw[-{Stealth[length=2.5mm]}] (0,1)--(0.93,1);
   \draw[-{Stealth[length=2.5mm]}] (2,1)--(1.07,1);
       \draw[-{Stealth[length=2.5mm]}] (1,1)-- (1,1.93);
 \draw[-{Stealth[length=2.5mm]}] (2,1)-- (2,1.93);
  \draw[-{Stealth[length=2.5mm]}] (1,2)--(1.93,2);
  \draw[-{Stealth[length=2.5mm]}] (3,2)--(3.93,2);
   \draw[-{Stealth[length=2.5mm]}] (3,2)--(2.07,2);
  \draw[-{Stealth[length=2.5mm]}] (3,1)--(2.07,1);
  \draw[-{Stealth[length=2.5mm]}] (3,1)--(3.93,1);
   \draw[-{Stealth[length=2.5mm]}] (4,1)-- (4,1.93);
    \draw[-{Stealth[length=2.5mm]}] (3,2)--(3,1.07);
\draw[-{Stealth[length=2.5mm]}] (3,3)--(3,2.07);
\draw[-{Stealth[length=2.5mm]}] (4,3)--(4,2.07);
  \draw[-{Stealth[length=2.5mm]}] (4,3)-- (3.07,3);

    \draw[fill=black] (0,0) circle (0.08);
\draw[fill=white] (1,0) circle (0.08);
\draw[fill=black] (1,1) circle (0.08);
\draw[fill=white] (0,1) circle (0.08);
\draw[fill=black] (2,0) circle (0.08);
\draw[fill=white] (2,1) circle (0.08);
\draw[fill=black] (3,1) circle (0.08);
\draw[fill=white] (4,1) circle (0.08);
\draw[fill=black] (4,2) circle (0.08);
\draw[fill=black] (3,3) circle (0.08);
\draw[fill=white] (4,3) circle (0.08);

\draw[fill=white] (1,2) circle (0.08);
\draw[fill=black] (2,2) circle (0.08);
\draw[fill=white] (3,2) circle (0.08);
\end{scope}

	\end{tikzpicture}
\caption{A colored snake graph as a bipartite graph
with a Kasteleyn orientation.}
\label{SnakeBip}
\end{figure} 
\end{ex}

Theorem~\ref{SnakeThm} implies the following statement.

\begin{cor}
\label{KastelDet}
The determinant of the Kasteleyn matrix of the snake graph $\G_{\frac{r}{s}}$
equipped with the weight system from Definition~\ref{MainDefn}
and orientation from Definition~\ref{KastOr} coincides
(up to a scalar multiple) with the polynomial $\cR(q)$ in the numerator of~$\left[\frac{r}{s}\right]_q$.
\end{cor}

\subsection{Numbering of vertices and $4$-diaginal Kasteleyn matrices}

Clearly, the determinant of the Kasteleyn matrix does not depend on the numeration
of black and white vertices 
(indeed, different choices lead to renumeration of basis vectors and.transposed Kasteleyn matrices).
However, in the case of snake graphs there is a way to number the vertices in such a way
that the Kasteleyn matrix becomes $4$-diagonal 
and has a form which is very similar to that of continuants
of continued fractions.

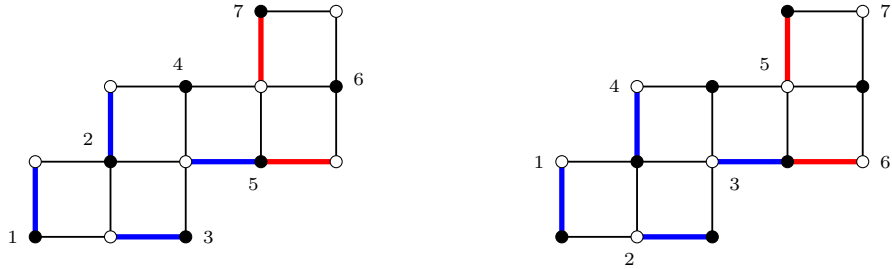
\begin{figure}[H]
	\centering
	\begin{tikzpicture}[scale=1]
 \begin{scope}[shift={(-4,0)}]
    \draw[line width=0.7pt] (0,0)-- (1,0);
    \draw[line width=2pt,blue] (1,0)-- (2,0);
    \draw[line width=2pt,blue] (0,0)-- (0,1);
    \draw[line width=0.7pt] (0,1)--(2,1)--(2,0);
    \draw[line width=0.7pt] (1,1)--(1,0);
     \draw[line width=2pt,red] (3,1)-- (4,1);
    \draw[line width=0.7pt] (4,1)--(4,2);
 \draw[line width=0.7pt] (3,2)--(4,2);
\draw[line width=0.7pt] (3,3)--(4,3);
\draw[line width=0.7pt] (4,2)--(4,3);
 \draw[line width=2pt,red] (3,2)-- (3,3);
     \draw[line width=2pt,blue] (1,1)--(1,2);
    \draw[line width=0.7pt] (2,1)--(2,2);
    \draw[line width=0.7pt] (1,2)--(3,2);
     \draw[line width=2pt,blue] (2,1)-- (3,1);
    \draw[line width=0.7pt] (3,1)--(3,2);
    \draw[fill=black] (0,0) circle (0.08);

\draw[fill=white] (1,0) circle (0.08);
\draw[fill=black] (1,1) circle (0.08);
\draw[fill=white] (0,1) circle (0.08);
\draw[fill=black] (2,0) circle (0.08);
\draw[fill=white] (2,1) circle (0.08);
\draw[fill=black] (3,1) circle (0.08);
\draw[fill=white] (4,1) circle (0.08);
\draw[fill=black] (4,2) circle (0.08);
\draw[fill=black] (3,3) circle (0.08);
\draw[fill=white] (4,3) circle (0.08);

\draw[fill=white] (1,2) circle (0.08);
\draw[fill=black] (2,2) circle (0.08);
\draw[fill=white] (3,2) circle (0.08);

\node at (-0.3, 0)  {$\scriptstyle{1}$};
\node at (0.7, 1.3)  {$\scriptstyle{2}$};
\node at (2.3, 0)  {$\scriptstyle{3}$};
\node at (1.9, 2.3)  {$\scriptstyle{4}$};
\node at (2.9, 0.7)  {$\scriptstyle{5}$};
\node at (4.3, 2.1)  {$\scriptstyle{6}$};
\node at (2.7, 3)  {$\scriptstyle{7}$};

\end{scope}

	 \begin{scope}[shift={(3,0)}]
	   \draw[line width=0.7pt] (0,0)-- (1,0);
    \draw[line width=2pt,blue] (1,0)-- (2,0);
    \draw[line width=2pt,blue] (0,0)-- (0,1);
    \draw[line width=0.7pt] (0,1)--(2,1)--(2,0);
    \draw[line width=0.7pt] (1,1)--(1,0);
     \draw[line width=2pt,red] (3,1)-- (4,1);
    \draw[line width=0.7pt] (4,1)--(4,2);
 \draw[line width=0.7pt] (3,2)--(4,2);
\draw[line width=0.7pt] (3,3)--(4,3);
\draw[line width=0.7pt] (4,2)--(4,3);
 \draw[line width=2pt,red] (3,2)-- (3,3);
     \draw[line width=2pt,blue] (1,1)--(1,2);
    \draw[line width=0.7pt] (2,1)--(2,2);
    \draw[line width=0.7pt] (1,2)--(3,2);
     \draw[line width=2pt,blue] (2,1)-- (3,1);
    \draw[line width=0.7pt] (3,1)--(3,2);
    \draw[fill=black] (0,0) circle (0.08);

\draw[fill=white] (1,0) circle (0.08);
\draw[fill=black] (1,1) circle (0.08);
\draw[fill=white] (0,1) circle (0.08);
\draw[fill=black] (2,0) circle (0.08);
\draw[fill=white] (2,1) circle (0.08);
\draw[fill=black] (3,1) circle (0.08);
\draw[fill=white] (4,1) circle (0.08);
\draw[fill=black] (4,2) circle (0.08);
\draw[fill=black] (3,3) circle (0.08);
\draw[fill=white] (4,3) circle (0.08);

\draw[fill=white] (1,2) circle (0.08);
\draw[fill=black] (2,2) circle (0.08);
\draw[fill=white] (3,2) circle (0.08);

\node at (-0.3, 1)  {$\scriptstyle{1}$};
\node at (0.9, -0.3)  {$\scriptstyle{2}$};
\node at (2.3, 0.7)  {$\scriptstyle{3}$};
\node at (0.7, 2)  {$\scriptstyle{4}$};
\node at (2.7, 2.3)  {$\scriptstyle{5}$};
\node at (4.3, 1)  {$\scriptstyle{6}$};
\node at (4.3, 3)  {$\scriptstyle{7}$};
	 \end{scope}

	\end{tikzpicture}
\caption{Numbering of black and white vertices of the snake graph $\G_{\frac{13}{3}}$.}
\label{NVFig}
\end{figure}

\begin{ex}
The $q$-deformed rational corresponding to the snake graph 
in Figures~\ref{SnakeBip} and~\ref{NVFig} is
$$
\left[\frac{13}{3}\right]_q=
\frac{1+2q+3q^2+3q^3+2q^4+q^5+q^6}{1+q+q^2},
$$
which is easy to calculate using the formulas from Section~\ref{RatSec}.

On the other hand, the Kasteleyn matrix associated to this snake graph is as follows
$$
M_{\frac{13}{3}}=
\begin{pmatrix}
q&-1&0&0&0&0&0\\[2pt]
-1&-1&q&-1&0&0&0\\[2pt]
0&q&0&-1&0&0&0\\[2pt]
0&0&-1&-1&-1&0&0\\[2pt]
0&0&0&q&-1&q^{-1}&0\\[2pt]
0&0&0&0&q^{-1}&0&-1\\[2pt]
0&0&0&0&-1&-1&-1
\end{pmatrix}
$$
and its determinant is 
$$
\det(M_{\frac{13}{3}})=q^{-2}+2q^{-1}+3+3q+2q^2+q^3+q^4=q^{-2}\,\cR_{\frac{13}{3}}(q),
$$
in accordance with Theorem~\ref{SnakeThmBis}.
\end{ex}

\begin{rem}
By elementary transformations, the obtained matrices are conjugate to $3$-diagonal matrices
that have the same determinant.
For instance, in the last example we have 
$$
\det(M_{\frac{13}{3}})=
\left|
\begin{array}{ccccccc}
q&-1&0&0&0&0&0\\[2pt]
-1&-[2]_q&q&0&0&0&0\\[2pt]
0&q&0&-1&0&0&0\\[2pt]
0&0&-1&-1&-1&0&0\\[2pt]
0&0&0&q&-[2]_{q^{-1}}&q^{-1}&0\\[2pt]
0&0&0&0&q^{-1}&0&-1\\[2pt]
0&0&0&0&0&-1&-1
\end{array}
\right|,
$$
where $[2]_q=1+q$.
\end{rem}

\subsection{Kasteleyn determinants for the $q$-deformed Fibonacci sequence}\label{FiKast}

The example of this section can also be found in~\cite{OveP} (see Remark 6).
The snake graphs corresponding to the rationals 
$\frac{F_{n+1}}{F_n}$ considered in Section~\ref{FiboSec}
are vertical strips of $n-1$ boxes
\begin{figure}[H]
	\centering
\begin{tikzpicture}[scale=0.8]
   
    \draw[line width=0.7pt] (0,0)--(1,0);
    \draw[line width=2pt,blue] (0,0)-- node[left]{$q$}(0,1);
    \draw[line width=2pt,red] (0,1)--node[left]{$q^{-1}$}(0,2);
    \draw[line width=2pt,blue] (0,2)--node[left]{$q$}(0,3);
     \draw[line width=2pt,red] (0,3)--node[left]{$q^{-1}$}(0,4);
    
    \draw[line width=0.7pt] (1,0)--(1,4);
    \draw[line width=0.7pt] (0,3)--(1,3);
    \draw[line width=0.7pt] (0,1)--(1,1);
    \draw[line width=0.7pt] (0,2)--(1,2);
    \draw[line width=0.7pt] (0,4)--(1,4);
    \draw[line width=0.7pt] (0,6)--(1,6);
    \draw[dashed] (1,4)--(1,6);
    \draw[dashed] (0,4)--(0,6);
      
    \draw[-{Stealth[length=2.5mm]}] (3,3.5)--(3,6);
    
    \node    at    (3,3)        {$n-1$};
    
    \draw[-{Stealth[length=2.5mm]}] (3,2.5)--(3,0);
    
     \draw[fill=black] (0,0) circle (0.08);
\draw[fill=white] (1,0) circle (0.08);
\draw[fill=black] (1,1) circle (0.08);
\draw[fill=white] (0,1) circle (0.08);
 \draw[fill=black] (0,2) circle (0.08);
\draw[fill=white] (1,2) circle (0.08);
\draw[fill=black] (1,3) circle (0.08);
\draw[fill=white] (0,3) circle (0.08);
\draw[fill=black] (0,4) circle (0.08);
\draw[fill=white] (1,4) circle (0.08);

	\end{tikzpicture}
\caption{Fibonacci vertical snakes.}
\label{FiboSnake}
\end{figure} 

\noindent
Graps of this type can be called snake graphs of $A$-type.
The corresponding Kasteleyn determinant is the following $n\times n$ determinant
$$
\det\Big(M_{\frac{F_{n+1}}{F_n}}\Big)
=
\left|
\begin{array}{ccccccccc}
1&\;1&&&\\[4pt]
-q&\;1&-q^{-1}&&\\[4pt]
&\;1&1&1&&\\[4pt]
&&-q&1&-q^{-1}&&\\[4pt]
&&&\ddots&\ddots&\!\!\ddots&\\[4pt]
&&&&-q&1&-q^{-1}\\[6pt]
&&&&&1&1
\end{array}
\right|
$$
(where $n$ is odd, it ends by $(-q\;,1)$ when $n$ is even).
It counts the number of weighted perfect matchings 
in the Fibonacci snake graph in Figure~\ref{FiboSnake}.

According to Theorem~\ref{SnakeThmBis}, the polynomial in the numerator of
$
\left[\frac{F_{n+1}}{F_{n}}\right]_q
$
is given by the above determinant, up to a scalar factor.
More precisely, taking the scalar factor~\eqref{Len} into account, 
we have the following corollary of Theorem~\ref{SnakeThmBis}.
$$
\tilde\F_{n+1}(q)
=
\left|
\begin{array}{ccccccccc}
1&\;1&&&\\[4pt]
-q^2&\;q&-1&&\\[4pt]
&\;1&\;1&1&&\\[4pt]
&&-q^2&q&-1&&\\[4pt]
&&&\ddots&\ddots&\!\!\ddots&\\[4pt]
&&&&-q^2&q&-1\\[6pt]
&&&&&1&\;1
\end{array}
\right|.
$$
This determinant coincides (up to some signs) 
with the continuant for the continued fraction~\eqref{FiboCF}.

\begin{rem}
The following continued fraction expansion
\begin{equation}
\label{TrueGoldCF}
\left[\varphi\right]_q\;=\;
\cfrac{1}{1-
          \cfrac{q^2}{1+
          \cfrac{q}{1-
           \cfrac{q^2}{1+\cfrac{q}{\ddots}}}}} 
\end{equation}
 was found in~\cite{OP}.
It is quite different from~\eqref{FiboCF} and satisfies some remarkable properties.
It is interesting to understand if there is a weight system and orientation on the
(finite or infinite) $A$-type snake graphs in Figure~\ref{FiboSnake}
such that the corresponding Kasteleyn determinant coincides with the continuant
determinant of the continued fraction~\eqref{TrueGoldCF}.
\end{rem}

\bigskip

\noindent{\bf Acknowledgments}.
We are grateful to Sam Evans, Perrine Jouteur, Ralf Schiffler, and especially
Sophie Morier-Genoud for enlightening discussions.


\bigskip

\noindent
{Valentin Ovsienko,
Centre National de la Recherche Scientifique,
Laboratoire de Math\'e\-ma\-tiques,
Universit\'e de Reims Champagne Ardenne,
Moulin de la Housse - BP 1039,
51687 Reims cedex 2,
France},\\
{Email:  valentin.ovsienko@univ-reims.fr}

\end{document}